\documentclass[11pt]{article}
\usepackage{amsmath}
  \usepackage{paralist}
  \usepackage{graphics}
  \usepackage{epsfig}
  \usepackage{float}
  \usepackage{graphicx}
  \usepackage{abstract}
 \usepackage{epstopdf}
 \usepackage{amssymb}
 \usepackage{cite}
 \usepackage{mathrsfs}
 \usepackage{amsthm}
 \usepackage[colorlinks=true]{hyperref}
\hypersetup{urlcolor=blue, citecolor=red}

\numberwithin{equation}{section}

\newtheorem{theorem}{Theorem}[section]

\newtheorem{lemma}[theorem]{Lemma}

\newtheorem{physical conclusion}[theorem]{Physical Conclusion}
\newtheorem{remark}[theorem]{Remark}

\newtheorem{definition}[theorem]{Definition}

\title{Mathematical Theory for Quantum Phase Transitions}

\author{ Tian Ma\thanks{Email:matian56@sina.com.}\ \ Da-Peng Li\thanks{Email:davidlee1234@163.com.} \ \
 Ruikuan Liu\thanks{Email:liuruikuan2008@163.com. 
} \ \  Jiayan Yang\thanks{Email:jiayan\_{}1985@163.com;} \ \ 
\\ \footnotesize $^{*,\dag,\ddag,\S}$Department of Mathematics, \footnotesize Sichuan University
 \footnotesize Chengdu, Sichuan 610064, China
 \\ \footnotesize $^\S$Department of Mathematics and Information Technology,~Southwest Medical University \\ \footnotesize Chengdu,
Sichuan 646000, China}

\begin{document}
\date{September 30, 2016}
\maketitle
\begin{abstract}
 Quantum Phase Transition (QPT) is a phase transition between different quantum states by adjusting some control parameters. Based on the Principle of Hamilton Dynamics (PHD) and the Principle of Lagrangian Dynamics (PLD), a general QPT model is established. Also, a definition of QPT is given. The important point is that the QPT model is applied to Bose-Einstein Condensate (BEC) and the corresponding bifurcation solutions (i.e., the quantum states) are obtained by using steady state bifurcation theory.
\begin{center}
\textbf{\normalsize keywords}
\end{center}
Quantum Phase Transition, ~Quantum Hamilton System, ~Steady State Bifurcation Theory, ~Bose-Einstein Condensate.
\end{abstract}

\section{Introduction}
It is well known that Quantum Phase Transition (QPT) is an important topic in condensed matter physics. The QPT is a phase transition between different quantum states by varying some control parameters, such as magnetic field or pressure. In particular, the first experiment related to QPT was performed by Anderson, Ensher, et al. \cite{AE} and Davis, Mewes, et al. \cite{DM} for Bose-Einstein Condensate in 1995. Since then, more quantum phase transitions were discovered including the ferromagnetic transition  in 1996 \cite{BR}, the experiment on ultra cold atoms in optical lattices  in 2002 \cite{GM}, and the experiment on crystals of $CoNb_{2}O_{6}$ by Coldea and collaborators in 2010 \cite{CTW}.


Notice that the classical definition of QPT originates from some physical experiments. In general, the main characters of the classical definition of QPT are described in the following three aspects.
\begin{description}
\item~(1)~The QPT occurs near the temperature of absolute zero;
\item~(2)~The QPT is accessed by varying some non-thermal physical parameters such as chemical composition or pressure;
\item~(3)~The QPT occurs as a result of competing quantum states by quantum fluctuations.
\end{description}

In nature, it is known that there are two different kinds of phase transitions including the thermal phase transition and the quantum phase transition (QPT). The thermal phase transition (i.e. equilibrium phase transition) occurs in the dissipative system. In this paper, we consider QPT as the phase transition in an energy conservation quantum system basing on the following properties:
\begin{description}
\item~(1)~QPT is the transition between different quantum states, which are the macroscopic performances caused by collective quantum behavior. As a matter of fact, the quantum state can be considered as the quantum behavior of a single particle or several particles, which can be characterized by a family of wave functions,
\begin{equation}\label{101}
    \psi=(\psi_{1},\psi_{2},\cdots,\psi_{k}),
\end{equation}
  where $\psi$ is a complex function
\begin{align}\label{102}
\psi_j=\psi_j^1+i \psi_{j}^2, \quad{1\leq j\leq k}.
\end{align}
For example, the quantum state is characterized by a single wave function $\psi$ in the scalar Bose-Einstein Condensate, and the quantum state is characterized by three wave functions $\psi=(\psi_{-1},\psi_{0},\psi_{1})$ in the spinor Bose-Einstein Condensate.

\item~(2)~ Based on the quantum mechanics theory, the system is characterized by a Hamiltonian $H(\psi,\lambda)$ with the wave function $\psi$ satisfying the following Sch$\ddot{o}$rdinger equation,
\begin{equation}\label{2003}
 i\hbar\frac{\partial \psi}{\partial t}=\frac{\delta H(\psi,\lambda)}{\delta \psi^{*}},
\end{equation}
where $\hbar$ is the Plank constant, $i$ is the imaginary unit, $\lambda$ is the control parameter, $H(\psi,\lambda)$ is the total energy of the system, $\psi$ and its conjugate $\psi^*$ are given by (\ref{101})-(\ref{102}). Moreover,  $(\ref{2003})$ equals to the following equations by the Quantum Hamilton System Theory in \cite{MW6}.
\begin{eqnarray}\label{2004}
\begin{aligned}
& \frac{\partial \psi^1_{i}}{\partial t}&=& \alpha \frac{\delta H(\psi,\lambda)}{\delta \psi^2_{i}},\\
& \frac{\partial \psi^2_{i}}{\partial t}&=&-\alpha \frac{\delta H(\psi,\lambda)}{\delta \psi^1_{i}}.
\end{aligned}
\end{eqnarray}
where $\alpha$ is a constant. Since $(\ref{2004})$ is the Hamilton equation used to describe an energy-conservation system, we consider QPT as the phase transition in the energy-conservation quantum system,
\end{description}

The rigorous theory of equilibrium Phase Transition has been established by T. Ma and S. Wang in \cite{MW5}, such as the dynamic phase transition theory in PVT systems in \cite{MW3} and the dynamic phase transitions for ferromagnetic systems in \cite{MW4}. As we know, there is little mathematical theory related to QPT. The main objective of this paper is to establish a general QPT model based on two fundamental physical principles, which are Principle of Hamilton Dynamics (PHD) and Principle of Lagrangian Dynamics (PLD) in \cite{MW6}. Furthermore, a  definition of QPT is given. Due to the mathematical definition in our paper, it is clear to see that the quantum states are the bifurcation solutions in the QPT model. Therefore, the bifurcation theory can be applied to study QPT. Finally, the bifurcation solutions (quantum states) in the scalar BEC are obtained due to steady state bifurcation theory.

This article is organized as follows. In section 2, we will introduce two physical principles: PHD and PLD. We also introduce the steady bifurcation theory, which is a basic mathematical tool to study QPT. Section 3 devotes to establish a general model for QPT based on PHD and PLD, to give a definition for QPT. In section 4, we study the QPT in the scalar BEC and describe some  physical conclusions of QPT in section 5.

\section{Preliminaries}
\subsection{Fundamental  principles of physical  system}

Energy conservation is a universal law in physics, which implies that the Principle of Hamilton Dynamics (PHD) is a universal principle to describe all the conservation physical systems. In other words, the PHD in classical mechanics can be generalized to all energy conservation physical fields, which is shown as follows.

\begin{lemma}(\textbf{Principle of Hamilton Dynamics}~\cite{MW6})\quad  For any conservation physical system,  there are two sets of state functions
\begin{align}\nonumber
u=(u_{1},\cdots, u_{N}),~~and~~v=(v_{1},\cdots,v_{N}),
\end{align}
such the energy density $\mathcal{H}$ is a functional of $u$ and $v$:
\begin{align}\nonumber
\mathcal{H}=\mathcal{H}(u,v,\cdots,D^{m}u,D^{m}v),\quad m\geq 0.
\end{align}
The Hamiltonian (total energy) of the system is
\begin{align}\nonumber
H(u,v)=\int_{\Omega}\mathcal{H}(u,v,\cdots,D^{m}u,D^{m}v)\text{d}x,~~\Omega\subset \mathbb{R}^{3},
\end{align}
provided that the system is described by continuous fields.  Moreover,~the state function $u$ and $v$ satisfy the equations
\begin{eqnarray}\nonumber
\begin{aligned}
  & \frac{\partial u}{\partial t}=\alpha \frac{\delta H}{\delta v},\\
  & \frac{\partial v}{\partial t}=-\alpha \frac{\delta H}{\delta u},
\end{aligned}
\end{eqnarray}
where $\alpha$ is a constant.
\end{lemma}

The Principle of Lagrangian Dynamics in classical mechanical system can be generalized to all physical motion systems, which is stated as follows
\begin{lemma}(\textbf{Principle of Lagrangian Dynamics}~\cite{MW6})
For a physical motion system, there are functions
\begin{align}\nonumber
u=(u_{1},u_{2},\cdots,u_{N}),
\end{align}
which describe the states of this system, and there exists a functional of $u$, given by
\begin{align}\nonumber
L(u)=\int_{\Omega}\mathcal{L}(u,Du,\cdots,D^{m}u)dx,
\end{align}
then the state functions of this system satisfy the variational equation as follows:
\begin{align}\nonumber
\delta L(u)=0.
\end{align}
The functional $L$ is called the Lagrange action, and $\mathcal{L}$ is called the Lagrange density.
\end{lemma}

\subsection{Steady state bifurcation theory}

In order to obtain the bifurcation solutions, we need to introduce the steady state bifurcation theory, which was established by Ma and Wang in \cite{MW1}, \cite{MW2}.

Let $X_{1}$ and $X$ be two Hilbert Spaces, and $X_{1}\subset X$ be a dense and compact inclusion. Consider a parameter family of nonlinear operator equations
\begin{align}\label{226}
L_{\lambda}u+G(u,\lambda)=0,
\end{align}
where $L_{\lambda}:X_{1}\rightarrow X$ is a completely continuous field, and
\begin{equation}\label{2266}
G(u,\lambda)=o(\|u\|)
\end{equation}
is a $C^{r}$ mapping depending on the parameter $\lambda$. Let the eigenvalues(counting multiplicity) of $L_{\lambda}$ be
 given by $\{\beta_{1}(\lambda),\beta_{2}(\lambda),\cdots\}$ with $\beta_i(\lambda)\in\mathbb{R}^1(1\leq i\leq m)$ such that
\begin{align}\label{227}
 &\beta_{i}(\lambda)\left\{
   \begin{array}{ll}
   <0, \quad \text{if}~\lambda<\lambda_{0},
   \\ =0, \quad \text{if}~\lambda=\lambda_{0},~~\forall 1\leq i\leq m,
   \\ >0, \quad \text{if}~\lambda>\lambda_{0},
 \end{array}
\right.\\
&\beta_{j}(\lambda_{0})\neq 0, ~~~\text{for~~all} ~\forall~~ j\geq m+1.\label{228}
\end{align}

According to  the spectral decomposition theory in \cite{MW2}, let $\{ \omega_{1}(\lambda),\omega_{2}(\lambda),$ $\cdots \}$ be the generalized eigenvectors corresponding to  $\{\beta_{1}(\lambda),\beta_{2}(\lambda),\cdots\}$, and $E_{1}^{\lambda}=\text{span}\{\omega_{1}(\lambda),\omega_{2}(\lambda),\cdots,\omega_{n}(\lambda)\}$, then near $\lambda=\lambda_{0}$ the space $X_{1}$ can be decomposed into the following direct sum
\begin{align}\label{229}
X_{1}=E_{1}^{\lambda}\oplus E_{2}^{\lambda},
\end{align}
where $E_{2}^{\lambda}$ is the complement of $E_{1}^{\lambda}$.

Now, we give the Lyapunov-Schmidt procedure. For any $\lambda$ near $\lambda_{0}$, the linear operator $L_{\lambda}$ can be decomposed into $L_{\lambda}=L_{\lambda}^{1}\oplus L_{\lambda}^{2}$  such that
\begin{eqnarray}\nonumber
  && L_{\lambda}^{1}=L_{\lambda}\mid_{ E_{1}^{\lambda}}:E_{1}^{\lambda}\rightarrow E_{1}^{\lambda},\\
  \nonumber && L_{\lambda}^{2}=L_{\lambda}\mid_{ E_{2}^{\lambda}}:E_{2}^{\lambda}\rightarrow \widetilde{E}_2^\lambda.
\end{eqnarray}
Thus, near $\lambda_{0}$, $(\ref{226})$ can be equivalently written as
\begin{eqnarray}\label{232}
  \label{232} && L_{\lambda}^{1}v_{1}+P_{1}G(v_{1}+v_{2},\lambda)=0,\\
  \label{232} && L_{\lambda}^{2}v_{2}+P_{2}G(v_{1}+v_{2},\lambda)=0,
\end{eqnarray}
where $P_{1}:X \rightarrow E_{1}^{\lambda}$ and $P_{2}:X \rightarrow \widetilde{E}_2^\lambda$ are the canonical projections, $u=v_{1}+v_{2}$, $v_{1}\in E_{1}^{\lambda}$ and $v_{2}\in E_2^\lambda$. From $(\ref{228})$, $L_{\lambda}^{2}$ is a linear homeomorphism near $\lambda_{0}$. By the implicit function theorem in \cite{MW2}, there exists a solution $v_{2}=f(v_{1},\lambda)$ of $(\ref{232})$, which is called center manifold function. Then bifurcation equation $(\ref{226})$ degenerate to  the following equation
\begin{eqnarray}\label{234}
  L_{\lambda}^{1}v_{1}+P_{1}G(v_{1}+f(v_{1},\lambda),\lambda)=0.
\end{eqnarray}
where $v_{1}\in E_{1}^{\lambda}$ and $\text{dim} E_{1}^{\lambda}=m<\infty$.

\section{Quantum Phase Transition(QPT)}

\subsection{The general model for QPT}

Note that QPT occurs in a physical motion system.
By the Lemma 2.2, the Lagrange action of this system can be
\begin{align}\label{220}
L(\psi,\psi^{*})=\int_{0}^{T}\int_{\Omega}\mathcal{L}(\psi,\psi^{*})dxdt,
\end{align}
where $\psi=\psi^{1}+i\psi^{2}$ is the wave function  and $\psi^{*}$ is the complex conjugate of $\psi$.

Furthermore, we get the variational equation as follows
\begin{align}\label{221}
\frac{\delta}{\delta \psi^{*}}L(\psi,\psi^{*})=0.
\end{align}

From \cite{MW6}, the Lagrange action of a quantum system with Hamiltonian $H(\psi,\psi^{*},\eta)$ is given by
\begin{align}\label{222}
\mathcal{L}(\psi,\psi^{*})=i\hbar\psi^{*}\frac{\partial \psi}{\partial t}-\mathcal{H}(\psi,\psi^{*},\eta),
\end{align}
where $H(\psi,\psi^{*},\eta)=\int_{\Omega}\mathcal{H}(\psi,\psi^{*},\eta)\text{d}x$ and $\eta$ is a parameter.

 Combing  $(\ref{221})$ and  (\ref{222}), we have
\begin{align}\label{223}
i\hbar\frac{\partial \psi}{\partial t}-\frac{\delta{H}(\psi,\psi^{*},\eta)}{\delta \psi^{*}}=0.
\end{align}

According to \cite{MW6}, the QPT system is energy conservation. From the Lemma 2.1, (\ref{223}) is equivalent  to
\begin{eqnarray}\label{2244}
\begin{aligned}
& \frac{\partial \psi^1}{\partial t}&=&\alpha \frac{\delta H(\psi,\psi^{*},\eta)}{\delta \psi^2},\\
& \frac{\partial \psi^2}{\partial t}&=&-\alpha \frac{\delta H(\psi,\psi^{*},\eta)}{\delta \psi^1},
\end{aligned}
\end{eqnarray}
where $\psi=\psi^{1}+i\psi^{2}$ and  $\alpha$ is a constant.

Owing to the basic theory in quantum physics \cite{CTW}, the wave function $\psi$ in an energy-conservation quantum system can be given as
\begin{align}\label{224}
\psi(x,t)=e^{-i\lambda t}\varphi(x),
\end{align}
where $\lambda=\frac{E}{\hbar}$, $\varphi(x)$ is a complex function of $x\in\Omega$, and $H(\psi,\psi^{*},\eta)=H(\varphi,\varphi^{*},\eta)$.

From (\ref{223}) and (\ref{224}), we get
\begin{equation}\label{2006}
  \frac{\delta H(\varphi,\varphi^{*},\eta)}{\delta \varphi^{*}}=E\varphi=h\lambda\varphi,
\end{equation}
where $E$ is the total energy, $\eta$ is the control parameter. Meanwhile, it is obvious that ($\ref{2006}$) is independent of time.

 Based on the quantum mechanics theory, the quantum system is described by $(\ref{2006})$ with a constraint condition as follows:
\begin{align}\label{q1}
H(\varphi,\varphi^{*},\eta)=
\min_{\int_{\Omega}|\tilde{\varphi}|^{2}dx=\text{const.}}
H(\tilde{\varphi},\tilde{\varphi}^{*},\eta),
\end{align}
which is a physical requirement. The equation (\ref{2006}) under the constraint condition (\ref{q1}) is the general model to study the quantum phase transition.

Without loss of generality, we impose the following Dirichlet boundary value condition
\begin{equation}\label{b1}
 \varphi=0~~~\text{on}~~~~\partial\Omega.
\end{equation}
\subsection{Definition of  QPT}

Now, we give the definition of QPT.

\begin{definition} Let $\Gamma $ be an energy-conservation quantum system described by (\ref{2006}) with (\ref{b1}). We say quantum phase transition occurs at $\eta=\eta_{c}$ in system $\Gamma$ provided that some new quantum states appear around the critical point $\eta=\eta_{c}$ of (\ref{2006}) with (\ref{b1}), see Figure 1.
\end{definition}

\begin{figure}[H]
  \centering
  \includegraphics[width=10cm]{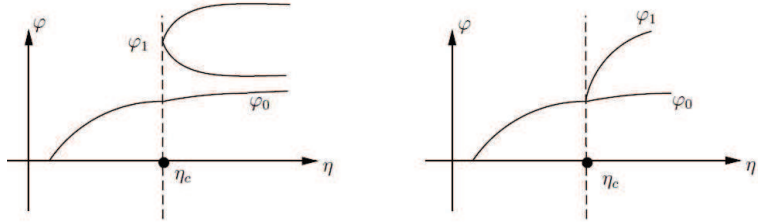}
  \caption{ The graph of definition 3.1. }\label{}
\end{figure}
\begin{remark}
In view of mathematics, it is clear to see that the new quantum states are the bifurcation solutions of (\ref{2006}) with (\ref{b1}). So, we give the equivalent definition as follows. Let $\varphi_0(\eta)$ be a quantum state of $\Gamma$  with (\ref{b1}), i.e., $\varphi_0(\eta)$ is a solution of (\ref{2006}) with (\ref{b1}).
We say quantum phase transition occurs at $\eta=\eta_{c}$ in $\Gamma$ provided that a new solution $\varphi_1(\eta)$ bifurcates from $(\varphi_0(\eta), \eta_c)$.  Hence, the bifurcation theory can be applied to study the quantum phase transitions.
\end{remark}

\section{Application}
\subsection{Bose-Einstein Condensate }

A Bose-Einstein condensate (BEC) is a state of matter of a dilute gas of bosons cooled to temperatures very close to absolute zero. Under such conditions, a large fraction of bosons occupy the lowest quantum state, at which point macroscopic quantum phenomena become apparent.


In our paper, the interaction between bosons is ignored, the Gross-Pitaevskii equation (GPE) provides a relatively good description of the behavior of atomic BEC's (see \cite{R}),  which is shown as follows
\begin{align}\label{302}
i\hbar\frac{\partial \psi}{\partial t}=-\frac{\hbar^{2}}{2m}\Delta\psi+V(x)\psi+g|\psi|^{2}\psi,
\end{align}
where $m$ is the mass of the bosons, $V(x)$ is the external potential, $\psi$ is a complex wave function, $g$ represents the inter-particle interactions(if $g<0$, particles in this system attract each other, if $g>0$, particles in this system repel each other), note that $\Omega\subset R^{3}$ is a bounded region.

In addition, without the loss of generality, let
\begin{align}\label{303}
 &V(x)=\left\{
   \begin{array}{ll}
   0, \quad \text{if}~x~\notin~\Omega,
   \\\beta, \quad\text{if}~x~\in~\Omega,
\end{array}
\right.
\end{align}
where $\beta\neq 0$. The Hamilton energy of (\ref{302}) is given by
\begin{align}\label{304}
H(\psi,\psi^*\beta)=\int_{\Omega}\left[\frac{\hbar^{2}}{4m}|\nabla\psi|^{2}+\beta|\psi|^{2}
+\frac{1}{2}g|\psi|^{4}\right]dx.
\end{align}
Hence, the general model (\ref{2006}) for BEC is obtained, which is given as
\begin{align}\label{305}
\left\{
\begin{aligned}
&-\frac{\hbar^{2}}{4m}\Delta\varphi+(\beta-\hbar\lambda)\varphi+g\varphi^{3}=0,~~~~x\in\Omega,\\
&\varphi=0,~~~~~~~~~~~~~~~~~~~~~~~~~~~~~~~~~~~~~~~~x\in\partial\Omega,
\end{aligned}
\right.
\end{align}
where $\varphi$ is a complex function of $x$, $\beta$ is the control parameter of this system, $\lambda=\frac{E}{\hbar}$ and $g\neq0$ is a given constant.

Physically speaking, $(\ref{305})$ has a constraint condition
\begin{align}\label{301}
N=\int_{\Omega}|\psi|^{2}dx=\text{const.},
\end{align}
where $\psi(x,t)=e^{-i\lambda t}\varphi(x)$ is the wave function of BEC,  $|\psi|^{2}$ is interpreted as the particle density.

\begin{remark}
Note that $\psi=\psi_{1}+i\psi_{2}$, the Hamilton equation of this system can be written as
\begin{eqnarray}\nonumber
\hbar\frac{\partial\psi_{1}}{\partial t}=\frac{\delta H}{\delta \psi_{2}}=-\frac{\hbar^{2}}{2m}\Delta\psi_{2}+(\beta+g|\psi|^{2})\psi_{2},\\
\nonumber \hbar\frac{\partial\psi_{2}}{\partial t}=-\frac{\delta H}{\delta \psi_{1}}=\frac{\hbar^{2}}{2m}\Delta\psi_{1}-(\beta+g|\psi|^{2})\psi_{1},
\end{eqnarray}
which equals to $(\ref{302})$, hence the system $(\ref{302})$ is a Quantum Hamilton System.
\end{remark}

\subsection{QPT in BEC}
Let $\Omega\subset \mathbb{R}^{3}$ be a bounded region.
\begin{align}\label{307}
X =L ^{2}(\Omega),\quad X_{1}=H ^{2}(\Omega)\cap H ^{1}_{0}(\Omega).
\end{align}
Define the linear completely continuous field $L_{\beta}:X_{1} \rightarrow X$
\begin{align}\label{309}
L_{\beta}=-\frac{\hbar^{2}\Delta}{4m}-\hbar\lambda+\beta.
\end{align}

Let $\{e_{k}\}$ be a basis of $X_{1}$,  $e_{k}$ and $\xi_{k}$ satisfy the following equations
\begin{eqnarray}\label{3099}
\left\{
   \begin{array}{ll}
   -\Delta e_{k}=\xi_{k}e_{k},
   \\e_{k}=0~on~\partial\Omega,
   \\ \int_{\Omega}e_{k}^{2}dx=1,
\end{array}
\right.
\end{eqnarray}
here $\xi_{k}$ is the k-th eigenvalue of $-\Delta$, $0<\xi_{1}\leq\xi_{2}\leq \cdots$, and $e_{k}$ is the eigenvector of $-\Delta$ corresponding to $\xi_{k}$.

For simplicity, we denote
\begin{eqnarray}
\label{312}&&\alpha_{k}=\int_{\Omega}e_{k}^{4}dx,\\
\label{3121}&&\gamma_{k}(\beta)=\frac{\hbar^{2}}{4m}\xi_{k}-\hbar\lambda+\beta.
\end{eqnarray}
Now, we give the following theorem, which is a characterization of BEC by regading the external potential $\beta$ as the control parameter.

\begin{theorem}
Let the external potential $\beta\neq0$ be a control parameter of $(\ref{305})$. Then the following conclusions of $(\ref{305})$  hold.
\begin{description}
\item~(1)~There is a family of quantum critical points $\beta_{k}$:
\begin{align}\label{310}
\beta_{k}=-(\frac{\hbar^{2}\xi_{k}}{4m}-\hbar\lambda),~~~k=1,2,\cdots
\end{align}
where $m$ is the mass of the bosons, $\xi_{k}$ is defied by (\ref{3099}), $\lambda=\frac{E}{\hbar}$(E is the total energy, $\hbar$ is the Plank constant).
\item~(2)~If $g>0$, then there are two nontrivial solutions $\psi_{k}^{\pm}$ bifurcates from  $(0,\beta_{k})$(k=1,2,$\cdots$) when $\beta<\beta_{k}$, which represent new quantum states.
\item~(3)~If $g<0$, then there are two nontrivial solutions $\psi_{k}^{\pm}$ bifurcates from  $(0,\beta_{k})$(k=1,2,$\cdots$) when $\beta>\beta_{k}$, which represent new quantum states.
\item~(4)~The nontrivial solutions $\psi_{k}^{\pm}$ are expressed  as follows:
\begin{align}\label{311}
\psi_{k}^{\pm}(x,t,\beta)=\bigg[\pm\sqrt{\frac{-\gamma_{k}(\beta)}{g\alpha_{k}}}e_{k}
+o(|\beta-\beta_{k}|)\bigg]e^{-i\lambda t},~~k=1,2,\cdots
\end{align}
\end{description}

\end{theorem}

\begin{proof}
The eigenvalues (counting multiplicity) $\{\gamma_{1}(\beta),\gamma_{2}(\beta),\cdots\}$ of (\ref{305}) are given by (\ref{3121}).

It is easy to check that
\begin{align}\label{317}
 &\gamma_{k}(\beta)\left\{
   \begin{array}{ll}
   <0, \quad \text{if}~~\beta<\beta_{k},
   \\=0, \quad \text{if}~~\beta=\beta_{k},
   \\>0, \quad \text{if}~~\beta>\beta_{k},
\end{array}
\right.\\
&\gamma_{j}(\beta_{k})\neq 0,\quad \text{for~ all} ~j\neq k.
\end{align}
where $\beta_k$ is defined by $(\ref{310})$. Due to the $\gamma_{k}(\beta)$ is single multiple,  $(\varphi,\beta)=(0,\beta_{k})$ is a bifurcation point of $(\ref{305})$ for each $k$ by the classical Krasnoselskii Theorem \cite{111}. Set
\begin{align}\label{}
G(\varphi,\beta)=g\varphi^{3}.
\end{align}
It is easy to see that the equation $(\ref{305})$ is equivalent to the following form:
\begin{align}\label{}
\left\{
\begin{aligned}
&L_{\beta} \varphi+G(\varphi,\beta)=0,~~~~x\in\Omega,\\
&\varphi=0,~~~~~~~~~~~~~~~~~~~~x\in\partial\Omega,
\end{aligned}
\right.
\end{align}
Notice that $E_{1}=\text{span}\{e_{k}\}$ and
\begin{align}\label{}
X_{1}= E_{1}\oplus E_{2}.
\end{align}
Let $L_{\beta}=L_{\beta}^{1}+L_{\beta}^{2}$,
\begin{align}\label{}
L_{\beta}^{1}=L_{\beta}|_{E_{1}}:E_{1}\longrightarrow E_{1}, \\
L_{\beta}^{2}=L_{\beta}|_{E_{2}}:E_{2}\longrightarrow \widetilde{E}_{2},
\end{align}
Let $P_{1}:X_{1}\rightarrow E_{1}$ be the canonical projection, and
\begin{align}\label{}
\varphi=\varphi_{1}+\varphi_{2},
\end{align}
where $\varphi_{1}\in E_{1}$ and $ \varphi_{2}\in E_{2}$, and $ \varphi_{1}=x_{k}e_{k}$.
Due to the Lyapunov-Schmidt reduction procedure \cite{MW2}, the reduction bifurcation equation of $(\ref{305})$  without center manifold function is obtained as follows
\begin{align}\label{}
L_{\beta}^{1}\varphi_{1}+P_{1}G(\varphi_{1},\beta)=0.
\end{align}
Hence, we have
\begin{align}\label{319}
\gamma_{k}(\beta)x_{k}+g\alpha_{k} x_{k}^{3}=0,
\end{align}
where
\begin{align}
\alpha_{k}=\int_{\Omega}e_{k}^{4}dx\neq0.
\end{align}
The bifurcation solutions around the critical points $\beta_k$ are obtained in two cases $g > 0$ and
$g < 0$.

First, we prove the case that $g>0$.

If $\beta<\beta_{k}$(i.e. $\gamma_{k}(\beta)<0$), there are two nontrivial solutions $x^{\pm}_{k}(\beta)$ of $(\ref{319})$ as follows
\begin{align}\label{}
x_{k}^{\pm}(\beta)=\pm\sqrt{\frac{-\gamma_{k}(\beta)}{g\alpha_{k}}}.
\end{align}
Then there are two new quantum states
\begin{align}\label{}
\varphi_{k}^{\pm}(x,\beta)=\pm\sqrt{\frac{-\gamma_{k}(\beta)}{g\alpha_{k}}}e_{k}+ o(|\beta-\beta_{k}|),
\end{align}
bifurcate from  $(0,\beta_{k})$ when $\beta<\beta_{k}$. Thus, the conclusion (2) is proved.

It is clear that $g<0$ is similar to the case $g>0$. Then the conclusion (3) is also proved. By (\ref{224}), we have
\begin{align}\label{324}
\psi_{k}^{\pm}(x,\beta,t)=\varphi_{k}^{\pm}(x,\beta)e^{-i\lambda t}.
\end{align}
The proof is complete.
\end{proof}

\begin{remark}
Notice that $g>0$ represents the particles repel
each other, the two bifurcation solutions  $\psi_{k}^{\pm}$  imply that  two new quantum states $\psi_{k}^{\pm}$ may occur around the critical point $\beta_{k}$ when $\beta<\beta_{k}$. However, $g<0$ represents the particles attract to each other, the two bifurcation solutions  $\psi_{k}^{\pm}$  imply that two new quantum states $\psi_{k}^{\pm}$ may occur around the critical point $\beta_{k}$ when $\beta<\beta_{k}$.
\end{remark}

In order to obtain some important properties of BEC, we only consider $\Omega=[0,L]\subset \mathbb{R}^{1}$. The main Theorem is shown as follows.

\begin{theorem}
Let $\Omega=[0,L]\subset \mathbb{R}^{1}$ and $g>0$. For the quantum system $(\ref{305})$,  then we have the following conclusions:
\begin{description}

\item~(1)~Let $\beta_{i}$ and $\beta_{j}$($i\neq j$) be two different bifurcation points. Let the nontrivial solution $\varphi_{k}(x,\beta)$ bifurcates from  $(0,\beta_{k})$ $(k=i,j)$, then there doesn't exit $\theta \in \mathbb{R}$, such that
\begin{align}\label{equal}
\varphi_{i}(x,\theta)=\varphi_{j}(x,\theta).
\end{align}
Namely, the two bifurcated branches, $\Gamma_i$ bifurcates from (0,$\beta_i$) and $\Gamma_j$ bifurcates from (0,$\beta_j)$, don't intersect.

\item~(2) if $\beta_{j+1}<\beta<\beta_{j}(j=1,2,\cdots)$, then there exit at least $2j$ new quantum states $\varphi(\beta)$ in BEC system $(\ref{305})$.


\end{description}
\end{theorem}

\begin{proof}
(1) Suppose on the contrary that there exits a point $\theta \in \mathbb{R}$, such that
\begin{align}\label{}
\varphi_{i}(x,\theta)=\varphi_{j}(x,\theta)(i\neq j).
\end{align}
Without loss of generality, we only consider $\varphi_{i}(x,\theta)$ and $\varphi_{i+1}(x,\theta)$. For simplicity, we consider $\varphi_{1}(x,\theta)$ and $\varphi_{2}(x,\theta)$ with,
\begin{align}\label{341}
\varphi_{1}(x,\theta)=\varphi_{2}(x,\theta).
\end{align}
From Theorem 3.1, we have
\begin{align}\label{323}
\varphi_{1}(x,\beta)=\sqrt{\frac{-\gamma_{1}(\beta)}{g\alpha_{1}}}\sin\bigg(\frac{\pi x}{L}\bigg)+o(|\beta-\beta_{1}|),\\
\varphi_{2}(x,\beta)=\sqrt{\frac{-\gamma_{2}(\beta)}{g\alpha_{1}}}\sin\bigg(\frac{2\pi x}{L}\bigg)+o(|\beta-\beta_{2}|),
\end{align}
When $|\beta-\beta_{1}|$ is small enough, the graph of $\varphi_{1}(x,\beta)$ is shown as follows.
\begin{figure}[H]
  \centering
  \includegraphics[width=5cm]{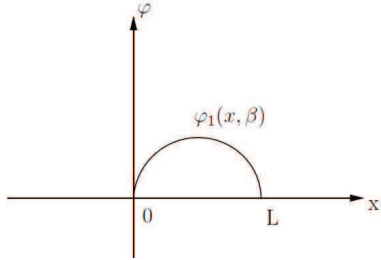}
  \caption{ The graph of $\varphi_{1}(x,\beta)$ with $|\beta-\beta_{1}|\ll1$. }\label{}
\end{figure}
When $|\beta-\beta_{2}|$ is small enough, the graph of $\varphi_{2}(x,\beta)$ is shown as follows.
\begin{figure}[H]
  \centering
  \includegraphics[width=5cm]{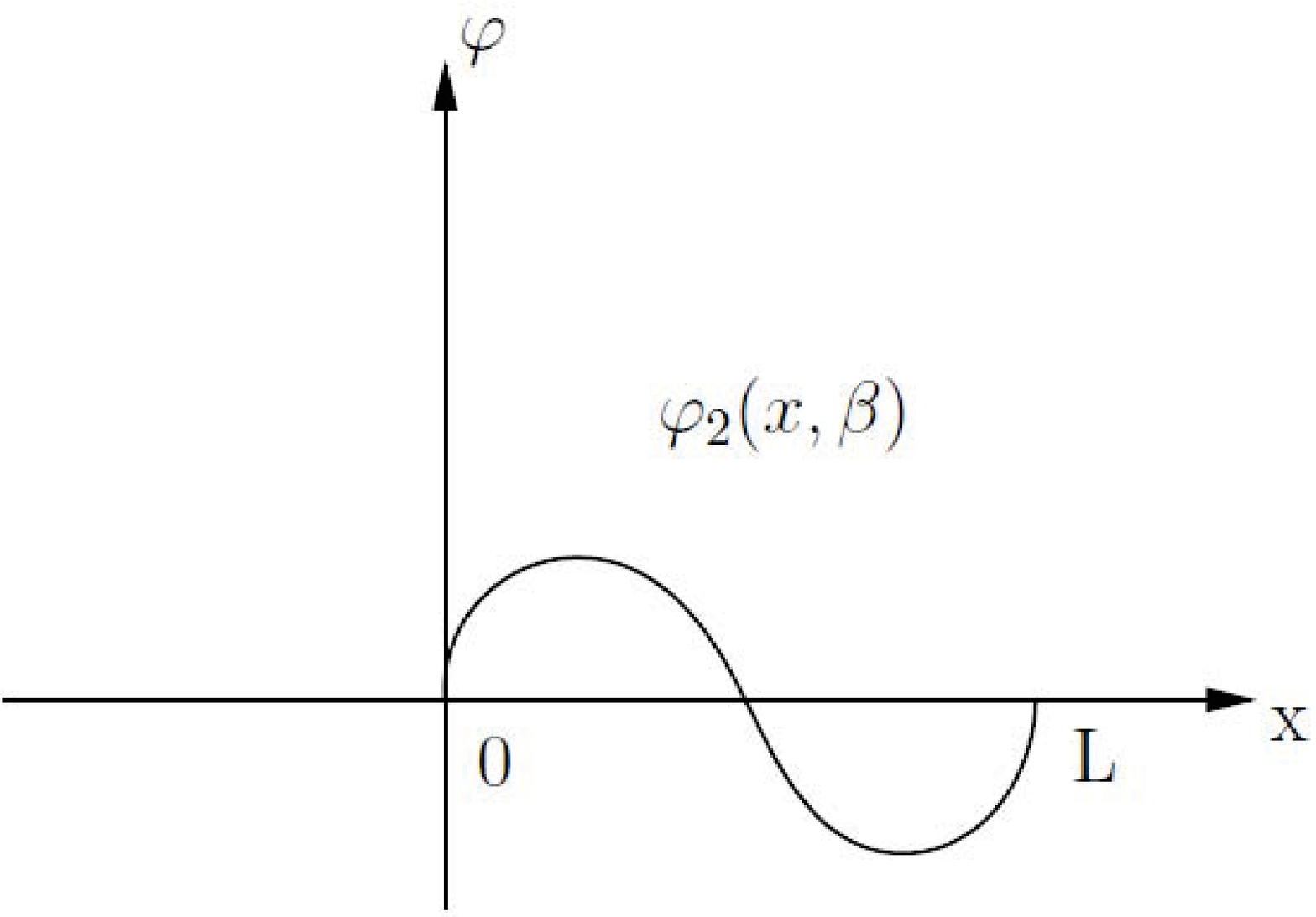}
  \caption{The graph of $\varphi_{2}(x,\beta)$ with $|\beta-\beta_{2}|\ll1$. }\label{}
\end{figure}
If the assumption $(\ref{341})$ holds, there exits a continuous function
\begin{align}\label{}
H:X_{1}\times[0,1]\rightarrow X_{1},
\end{align}
such that
\begin{align}\label{}
 \left\{
   \begin{array}{ll}
   H(\varphi,0)=\sqrt{\frac{-\gamma_{1}(\beta)}{g\alpha_{1}}}\sin(\frac{\pi x}{L}),
   \\H(\varphi,1)=\sqrt{\frac{-\gamma_{2}(\beta)}{g\alpha_{1}}}\sin(\frac{2\pi x}{L}),
\end{array}
\right.
\end{align}
where $X_1$ is defined by (\ref{307}).

It follows that the solution $\varphi_{2}(x,\beta)$ in Figure 2 is continuously deformed to the solution $\varphi_{1}(x,\beta)$ in Figure 1.

Next, we consider the following three cases.

\textbf{Case 1.}
There exists a point $\beta_{0}$, an interval $(\ell_{1}, \ell_{2})\subset[0,L]$, a point $x_{0}\in (\ell_{1}, \ell_{2})$, a point $\delta\in(0,1)$, and $\varphi(x,\beta_{0})$ satisfies
\begin{align}\label{347}
 \left\{
   \begin{array}{ll}
   \varphi(x,\beta_{0})=H(\varphi,\delta),
   \\ \varphi(\ell_{1},\beta_{0})=\varphi(\ell_{2},\beta_{0})=0,
   \\ \varphi(x_{0},\beta_{0})=0,\\
   \varphi(x,\beta_{0})\leq 0,~~\forall x\in (\ell_{1}, \ell_{2}),
\end{array}
\right.
\end{align}
which implies the graph of $\varphi(x,\beta_{0}) $ is shown as follows:
\begin{figure}[H]
  \centering
  \includegraphics[width=6cm]{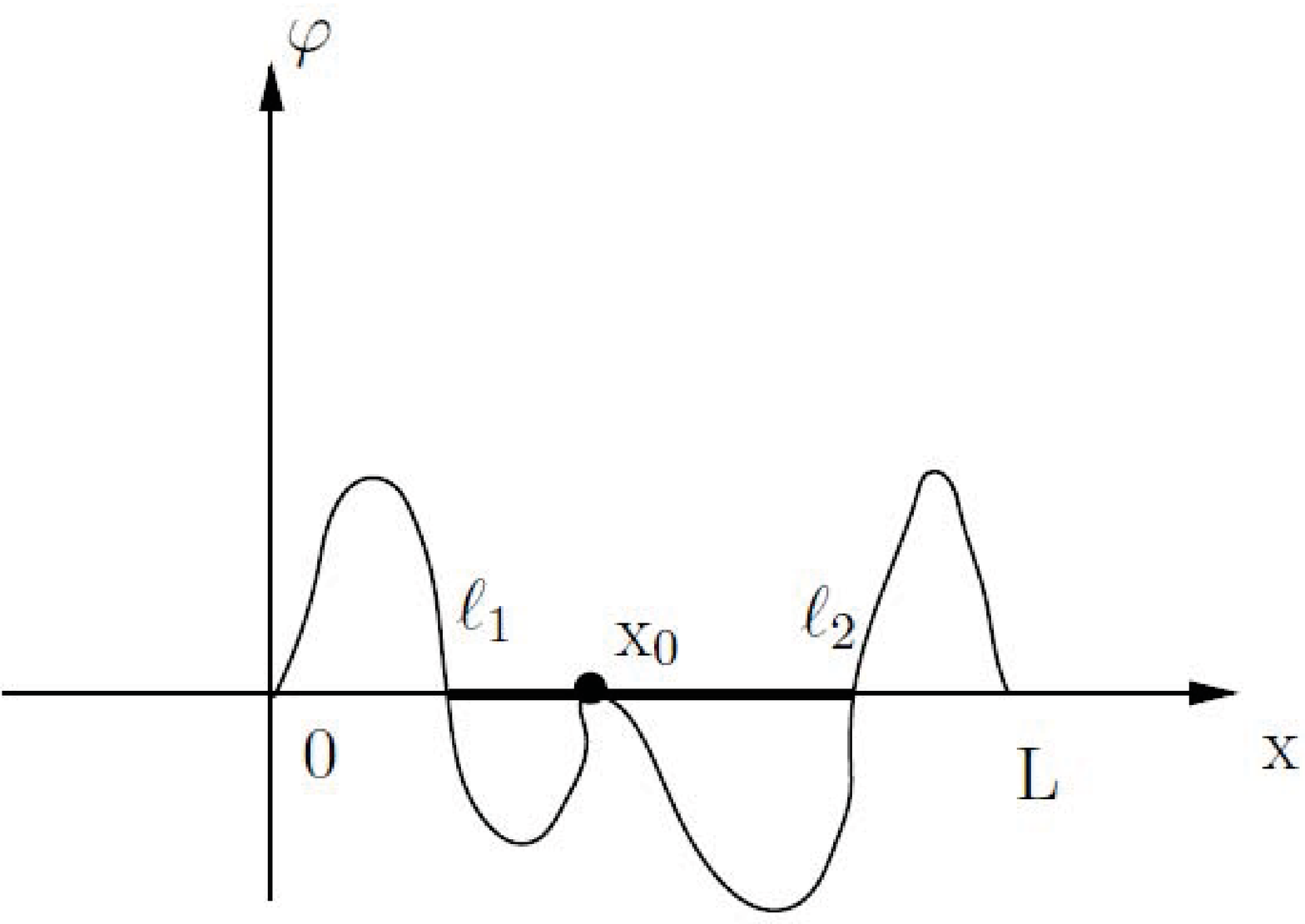}
  \caption{Figure of (\ref{347}).}\label{}
\end{figure}
Moreover, $\varphi(x,\beta_{0}) $  satisfies the following equation
\begin{align}\label{}
 \left\{
   \begin{array}{ll}
   -\frac{\hbar^{2}}{4m}\ddot{\varphi}+\beta_{0}\varphi-\hbar\lambda\varphi+g\varphi^{3}=0,~~x\in(\ell_{1},\ell_{2}),
   \\ \varphi(\ell_{1},\beta_{0})=\varphi(\ell_{2},\beta_{0})=0.
\end{array}
\right.
\end{align}
It is obvious to see that
\begin{align}\label{}
 \frac{\hbar^{2}}{4m}\ddot{\varphi}-\beta_{0}\varphi-\hbar\lambda\varphi=g\varphi^{3}\leq 0, ~~x\in(\ell_{1},\ell_{2}).
\end{align}
By the maximum principle \cite{GT}, we have
\begin{align}\label{}
 \varphi(x,\beta_{0})<0,~~\forall x\in(\ell_{1},\ell_{2}),
\end{align}
which is a contradiction.

\textbf{Case 2.}
There exists a point $\beta_{0}$, an interval $(\ell_{1}, L] \subset[0,L]$, a point $\delta \in (0,1)$, and the solution  $\varphi(x,\beta_{0})$ satisfies
\begin{align}\label{351}
 \left\{
   \begin{array}{ll}
   \varphi(x,\beta_{0})=H(\varphi,\delta),~~\forall x\in (0, L]\\
   \varphi(x,\beta_{0})=0,~~\forall x\in (\ell_{1}, L],
   \\ \varphi(x,\beta_{0})\geq 0,~~\forall x\in [0,\ell_{1}),
\end{array}
\right.
\end{align}
which is shown in the Figure 5
\begin{figure}[H]
  \centering
  \includegraphics[width=6cm]{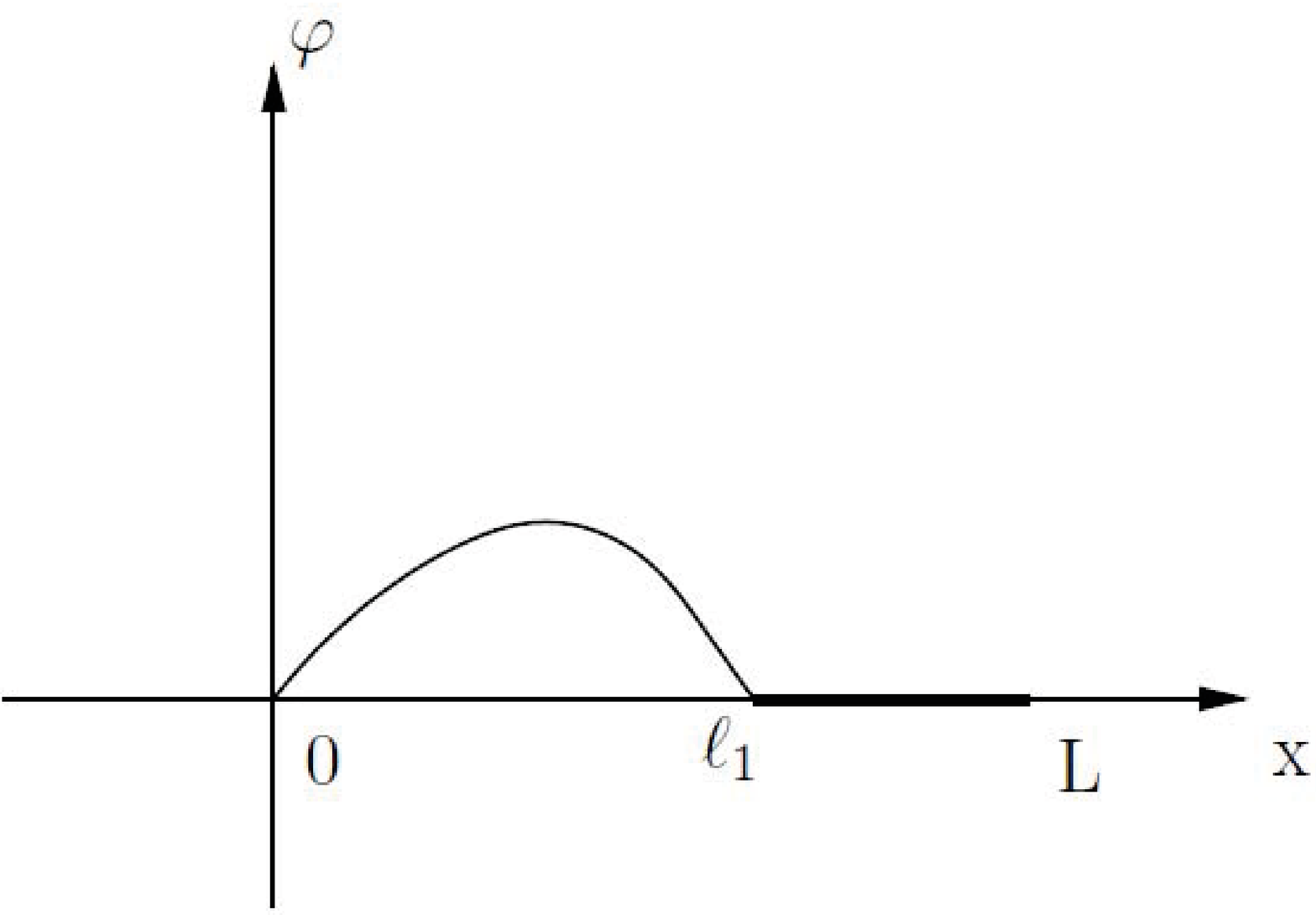}
  \caption{Figure of (\ref{351}).}\label{}
\end{figure}
Then we have
\begin{align}\label{}
 \frac{\hbar^{2}}{4m}\ddot{\varphi}-\beta_{0}\varphi-\hbar\lambda\varphi=g\varphi^{3}\geq 0, ~~\forall x\in [0,L].
\end{align}
From the maximum principle \cite{GT} , we have
\begin{align}\label{}
 \varphi(x,\beta_{0})> 0,~~\forall x\in [0,L],
\end{align}
which is a contradiction.

\textbf{Case 3.}
There exits an interval $(\ell,L]\subset[0,L]$, and a continuous function $\delta=\delta(\beta),~0<\delta<1 $, such that
\begin{align}\label{354}
 \left\{
   \begin{array}{ll}
   \varphi(x,\beta)=H(\varphi,\delta(\beta) ),\\
   \varphi(x,\beta)\leq 0,~~\forall x \in (\ell, L].
   \\ \varphi(x,\beta)\geq 0,~~\forall x\in [0,\ell).\\
   \ell \rightarrow L,~~as ~\varphi_{2}\rightarrow \varphi_{1}.\\
   \varphi(x,\beta)\rightarrow 0 ~~as ~\varphi_{2} \rightarrow \varphi_{1},x \in (\ell, L],
\end{array}
\right.
\end{align}
which is shown in the following Figure 6:
\begin{figure}[H]
  \centering
  \includegraphics[width=7cm]{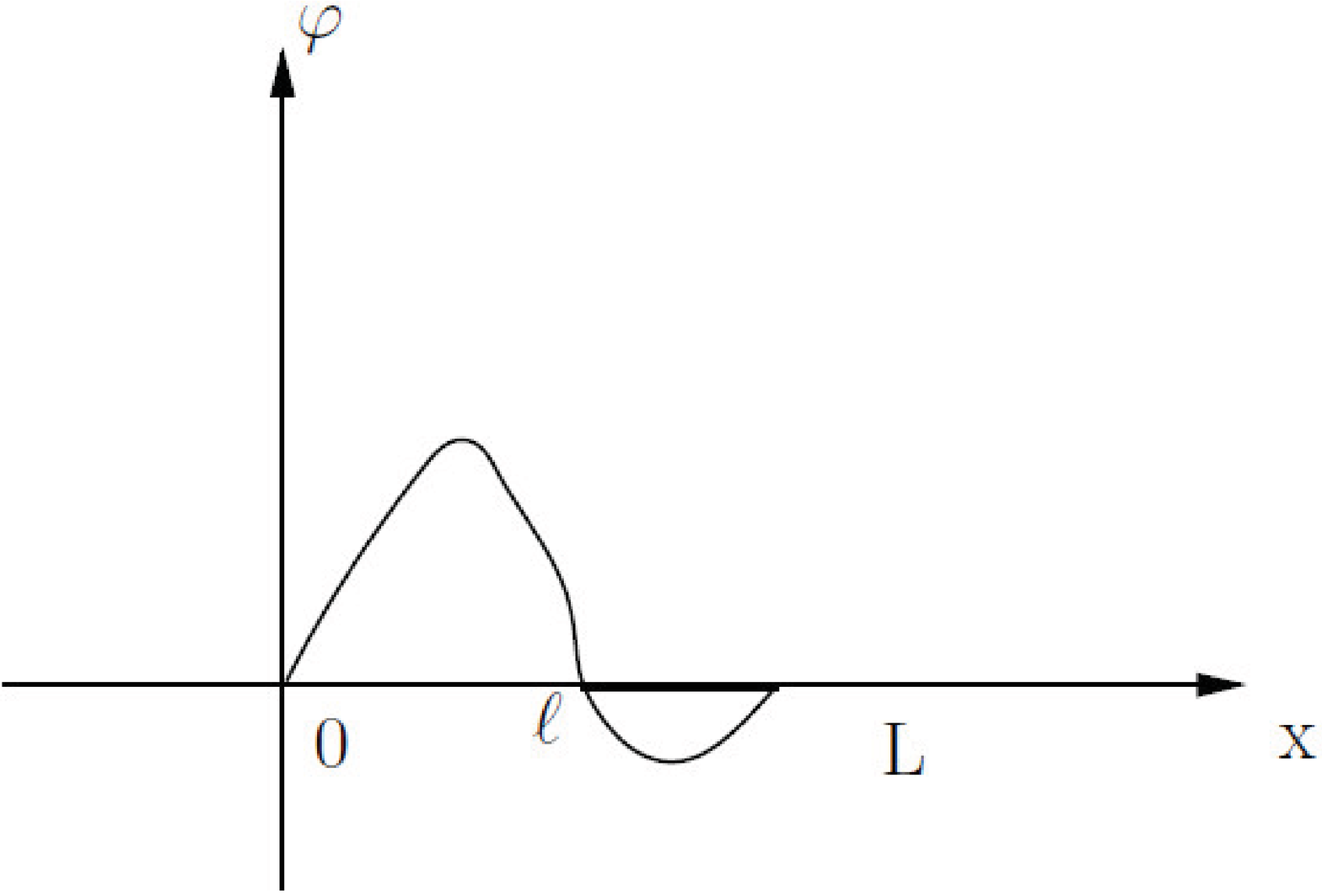}
  \caption{Figure of (\ref{354}).}\label{}
\end{figure}
Let $ \Lambda=(\ell,L] $. Consider the equation $(\ref{305})$ depending on $\Lambda$,
\begin{align}\label{}
\left\{
\begin{aligned}
&-\frac{\hbar^{2}}{4m}\Delta\varphi+(\beta-\hbar\lambda)\varphi+g\varphi^{3}=0,~~x\in \Lambda.\\
&\varphi=0,~~on~~\partial \Lambda.
\end{aligned}
\right.
\end{align}
Then we have
\begin{align}\label{}
&-\frac{\hbar^{2}}{4m}\Delta\varphi-\hbar\lambda\varphi+g\varphi^{3}=-\beta\varphi,~~x\in \Lambda.
\end{align}
By the application of Poincare inequality, it is easy to see that
\begin{align}\label{}
\begin{aligned}
&-\beta \parallel\varphi\parallel_{L^{2}}^{2}=<-\beta \varphi,\varphi>_{L^{2}}\\
&=<-\frac{\hbar^{2}}{4m}\Delta\varphi-\hbar\lambda\varphi+g\varphi^{3},\varphi>_{L^{2}}\\
&=\int_{\Lambda}\bigg[\frac{\hbar^{2}}{4m}(\nabla\varphi)^{2}-\hbar\lambda\varphi^{2}+g\varphi^{4}\bigg] dx\\
&\geq\frac{C}{L-\ell}\int_{\Lambda}\varphi^{2}dx.
\end{aligned}
\end{align}
It is shown that
\begin{align}\label{}
-\beta \geq \frac{C}{L-\ell}\rightarrow \infty~~as~~\varphi_{2} \rightarrow \varphi_{1},
\end{align}
which  contradicts against $\beta<\infty$.  The proof of (1) is complete.

(2) First, we prove that  there doesn't exist a point $\beta_{0}\in \mathbb{R}$, such that
\begin{align}\label{eq1}
\varphi(x,\beta)\rightarrow\infty, ~as~\beta\rightarrow\beta_{0}
\end{align}
for any bifurcation solution $\varphi(x,\beta)$ of $(\ref{305})$.

Suppose on the contrary that there exits a point $\beta_{0}\in \mathbb{R}$ and a bifurcation solution $\varphi(x,\beta)$, such that
\begin{align}\label{}
\varphi(x,\beta)\rightarrow\infty, ~as~\beta\rightarrow\beta_{0}.
\end{align}
Let $I=(\beta_{0}-\delta, \beta_{0}+\delta )$ be a neighborhood of the point $\beta_{0}$, and
\begin{align}\label{}
Q(\varphi,\beta)=-\frac{\hbar^{2}}{4m}\Delta\varphi+\beta\varphi-\hbar\lambda\varphi+g\varphi^{3}.
\end{align}
Then we have
\begin{align}\label{337}
\begin{aligned}
&<Q(\varphi,\beta),\varphi>_{L^{2}}\\
&=<-\frac{\hbar^{2}}{4m}\Delta\varphi+\beta\varphi-\hbar\lambda\varphi+g\varphi^{3},\varphi>_{L^{2}}\\
&=\int_{I}\bigg[\frac{\hbar^{2}}{4m}(\nabla\varphi)^{2}
+(\beta-\hbar\lambda)\varphi^{2}+g\varphi^{4}\bigg]dx\\
&\geq\int_{I}[(\beta-\hbar\lambda)\varphi^{2}+g\varphi^{4}]dx\rightarrow\infty,~~as~|\varphi| \rightarrow\infty.
\end{aligned}
\end{align}
Since $\varphi$ is a solution of $(\ref{305})$, we have $Q(\varphi,\beta)=0$, i.e.,
\begin{align}\label{338}
<Q(\varphi,\beta),\varphi>_{L^{2}}=0,
\end{align}
which contradicts against (\ref{337}). So, (\ref{eq1}) holds.

Combining conclusion (1), (\ref{eq1}) and the global Rabinowitz bifurcation theorem in \cite{RP}, it is easy to see that $\varphi(\beta)$ is continuous dependence on $\beta\rightarrow -\infty$. The proof is complete.

\end{proof}


\section{Physical conclusions}

From the Theorem 4.4, we have  the following two properties of the scalar BEC:
\begin{itemize}
\item There is a series of quantum critical points $\{\beta_{k}\}_{k=1}^{\infty}$ in $(\ref{305})$;

\item The number of quantum states in $(\ref{305})$ increases as $\beta_{k}\rightarrow -\infty$. Specifically, if $\beta_{k+1}<\beta<\beta_{k}$, then there exit at least $2k$ ($k\geq 1$) nontrivial quantum states in BEC system $(\ref{305})$, see Figure 7.
\end{itemize}


\begin{figure}[H]
  \centering
  \includegraphics[width=8cm]{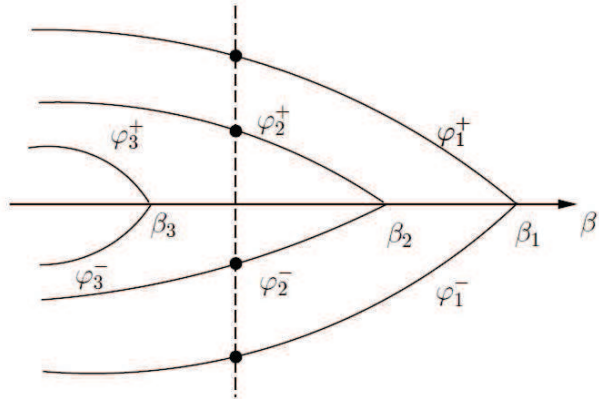}
  \caption{QPT for BEC}\label{}
\end{figure}


Moreover,  we summarize the following general properties of the QPT.

\begin{itemize}
  \item The QPT system is characterized by the QPT model (\ref{2006}), and there is a series of quantum critical points in the QPT systems, such quantum critical points are bifurcation points of (\ref{2006}). The new quantum states appear around the quantum critical points, which are the bifurcation solutions of (\ref{2006});
  \item The number of quantum states  increases as the control parameter be adjusted in (\ref{2006}), specifically, we may predict that the the number of quantum states in (\ref{2006}) become more and more as $\beta_{k}\rightarrow -\infty$.
\end{itemize}

In addition, these properties will play an important role in understanding the QPT and interpreting the corresponding experiments. The important point is that these general properties will make some important predictions which can be verified by the physical experiments.

 {\footnotesize


\begin{thebibliography}{90}

\bibitem{AM}  Ambrosetti. A.,  Malchiodi. A., Nonlinear analysis and semilinear elliptic problems. Cambridge University Press, Cambridge, 2007.

\bibitem{AE}  Anderson M. H., Ensher J. R., Matthews M. R., Wieman C. E., Cornell E. A., Observation of Bose-Einstein condensation in a dilute atomic vapor. Science 269, 1995, 198.

\bibitem{BR} Bitko. D., Rosenbaum. T. F, Aeppli. G., Phys. Rev. Lett, 77, 1996,940.

\bibitem{CTW}Coldea. R, Tennant. D. A, Wheeler. E. M et al. Science 327, 2010,177.

\bibitem{DM} Davis. K. B., Mewes. M. O., Andrews. M. R., van. Druten. N. J., Durfee. D. S., Kurn. D. M., Ketterle W., Bose-Einstein condensation in a gas of sodium atoms. Phys. Rev. Lett. 75, 1995, 39-69.


\bibitem{GM}Greiner. M., Mandel. O., Esslinger.T, H\"{a}nsh.T.W, Bloch.I.Nature 415, 2002, 39.
\bibitem{GT} Gilbarg. D., Trudinger. N. S., Elliptic partial differential equations of second order. Springer-Verlag, Berlin, 2001.
\bibitem{LL} Landau L. D., Lifshitz E. M., Quantum
Mechanics. Butterworth-Heinemann, Amsterdam, 2003.

\bibitem{MW1} T. Ma, S. Wang. Bifurcation of
 nonlinear equations: I.Steady State bifurcation.
  Method and Applications of Analysis. 11(2), 2004, 155-178.
\bibitem{MW2} T. Ma, S. Wang. Bifurcation Theory and Applications.
 World Scientific Publishing Co. Pte. Ltd.: Hackensack, NJ, 2005.

\bibitem{MW3} T. Ma, S. Wang. Dynamic phase transition theory in PVT systems. Indiana University Mathematics Journal 57(6), 2008, 2861-2889.
\bibitem{MW4} T. Ma, S. Wang. Dynamic phase transitions for ferromagnetic systems,  Journal of Mathematical Physics, 49, 2008, 1-18.
\bibitem{MW5} T. Ma, S. Wang. Phase transition dynamics. Springer-Verklag, 2013.
\bibitem{MW6} T. Ma, S. Wang. Mathematical Principles of Theoretical Physics.  Science Press, Beijing, 2015.

\bibitem{R} Rogel-Salazar J., The Gross-Pitaevskii equation and Bose-Einstein condensates. Eur. J. Phys. 34, 2013, 247-257.
\bibitem{RP} Rabinowitz P.H., Some global results for nonlinear eigenvalue problems, J. Funct. Anal. 7, 1971, 487-513.
\bibitem{RPH} Rabinowitz P. H., Some aspects of nonlinear eigenvalue problems, Rocky Mountain J. Math. 3 (2), 1973, 161-202.




\end{thebibliography}
\end{document}